\documentclass[a4paper,10pt,draft]{amsart}

\title[$\textup{m}$-microlocal ellipticity in $L^p$]{$\textup{m}$-microlocal elliptic pseudodifferential operators acting on  $L^p_\textup{loc}(\Omega)$}

\usepackage{cite}
\usepackage{amssymb}
\usepackage{latexsym}
\usepackage{amsmath}
\usepackage{mathrsfs}
\usepackage[X2,T1]{fontenc}
\usepackage{amsthm}
\usepackage{upgreek}
\usepackage{calc}

\def\aa{\`{a}\ }

\numberwithin{equation}{section}          
\newtheorem{thm}{Theorem}
\numberwithin{thm}{section}

\newcommand{\rubrik}{}
\newtheorem{prop}[thm]{Proposition}

\newtheorem{lemma}[thm]{Lemma}

\newtheorem{defn}[thm]{Definition}

\theoremstyle{remark}

\newtheorem{rem}[thm]{Remark}
\newtheorem{ex}[thm]{Example}

\author{Gianluca Garello}

\address{Dipartimento di Matematica Universit\aa di Torino\\
Via Carlo Alberto 10\\
I-10123, Torino,\\
Italy }

\email{gianluca.garello@unito.it}

\author{Alessandro Morando}

\address{DICATAM - Sezione di Matematica\\
Universit\`a di Brescia\\
Via Valotti 9\\
I-25133, Brescia\\
Italy}
\email{alessandro.morando@unibs.it}

\begin{document}



\begin{abstract} In the first part of the paper the authors study the minimal and maximal extension of a class of weighted pseudodifferential operators in the Fr\'echet space $L^p_\textup{loc}(\Omega)$.
In the second one non homogeneous microlocal properties are introduced and propagation of Sobolev singularities for solutions to (pseudo)differential equations is given.  For both the arguments actual examples are provided.

\end{abstract}

\maketitle
\section{introduction}\label{INT}
Let us fix the attention on the pseudodifferential operator with classical quantization defined by:
\begin{equation}\label{eqINT1}
a(x,D)u=(2\pi)^{-n}\int e^{i x\cdot\xi} a(x,\xi)\hat u(\xi)\, d\xi, \quad u\in C^\infty_0(\Omega),
\end{equation}
where $\hat u(\xi)$ is the Fourier transform of $u\in C^\infty_0(\Omega)$, $\Omega$ open subset of $\mathbb R^n$. The symbol $a(x, \xi)$ is considered  in the H\"ormander local symbol classes  $S^m_{\rho,\delta}(\Omega)$, $m\in \mathbb R$, $0\leq\delta<\rho\leq 1$.
The problem of  $L^p$ boundedness, $1<p<\infty$, when $\rho=1$ is completely solved from the beginning  of the pseudodifferential operators theory. The same is true for the study of the regularity of solutions to elliptic pseudodifferential equations.  \\
The matter is completely different when we consider $\rho< 1$. Feffermann in \cite{FE1}, 1973, proved that an operator $a(x,D)$ with  symbol in the class $S^{-m}_{\rho, \delta}$ is bounded into $L^p$ if $n(1-\rho)\left\vert \frac {1}{p}-\frac{1}{2} \right\vert\leq m<\frac{n}{2}(1-\rho)$, the result is sharp.\\
Also the generalized weighted symbol classes introduced by R. Beals in \cite{BE1}, \cite{BE2}, \cite{BE3} require that $\delta<\rho=1$ for obtaining  $L^p$ bounded zero order  pseudodifferential operators, see \cite[\S 1]{BE3}.\\
M. E. Taylor in \cite[\S XI]{TA1} introduces the subclass $M^m_\rho(\Omega)$ of the symbols $a(x,\xi)$ in $S^m_{\rho, 0}(\Omega)$ such that $\xi^\gamma\partial^\gamma_\xi a(x,\xi)\in S^m_{\rho,0}(\Omega)$, for any multi-index $\gamma\in\mathbb Z^n_+$ which has only components equal to zero or one. $M^0_\rho(\Omega)$ defines a class of $L^p$ bounded pseudodifferential operators which moreover satisfy local regularity properties when their symbols are elliptic, see \cite{TA1}, \cite {GM1}.\\
Results of $L^p$ continuity for pseudodifferential operators with symbols $a(x,\xi)$ in weighted Sobolev and Besov spaces, with respect to the $x$ variable, which satisfy suitable estimates at infinity on the  derivatives with respect to $\xi$, can be found in \cite{GM4}, \cite{GM5}, \cite{GM6}, togheter with the regularity for the respective weighted elliptic operators.\\
In the present paper, following the approach of Rodino \cite{RO1}, we introduce a class of local vector weighted symbols: $S_{m,\Lambda}(\Omega)$ where $m(\xi)$ and $\Lambda(\xi)=(\lambda_1(\xi), \dots, \lambda_n(\xi))$ are  positive continuous weight function and  weight vector. The corresponding pseudodifferential operators will be called in the following $m$-pseudodifferential operators and they could be considered in the frame of general pseudodifferential calculus of R. Beals \cite{BE1}, H\"ormander \cite{HO2}.
Thanks to the assumptions on the weight vector $\Lambda(\xi)$, the classes  $S_{m,\Lambda}(\Omega )$ satisfy a condition of Taylor type and the $m$-pseudodifferential operators are  $L^p$  continuous when the weight $m(\xi)$ is bounded, see \cite{GM1} and the next Theorem \ref{CONT}\\
Relying on these arguments  we construct  in \S\ref{MM} the minimal and maximal  extensions in the local Fr\'echet space $L^p_\textup{loc}(\Omega)$ for the $m$-pseudodifferential operators with $m(\xi)$ general unbounded weight, see \cite{WO2} for similar arguments in the global Banach space $L^p$. When  $a(x,\xi)$ is elliptic in generalized sense, with respect to $m(\xi)$, we prove that  the minimal and maximal extensions coincide and they have as domain the weighted local Sobolev space $H^p_{m, \textup{loc}}(\Omega)$.\\
The second aim of the paper is the study of the microlocal properties of $m$-pseudodifferential operators. From this point of view the main problem consists  in the complete lack of any homogeneity property of the weights $m(\xi)$ and $\Lambda(\xi)$. For this reason the set of the points where the symbol $a(x,\xi) \in S_{m,\Lambda}(\Omega)$ is not elliptic, say the $m$-characteristic set of $a(x,D)$, has not any conic properties with respect to  the $\xi$ variable. Thus for describing the microlocal Sobolev regularity of a distribution $u\in \mathcal D'(\Omega)$, we cannot use  the conic  neighborhoods as done in the classical definition of the H\"ormander wave front set, see \cite[I]{HO1}. Also the quasi-homogeneous neighborhoods introduced  by Lascar \cite {LA1}, se also \cite{GM3}, are not useful  in this case. For this reason in \S\ref{MP}, following the arguments in \cite{RO1}, \cite{GA92}, \cite{GM2} using a suitably defined neighborhood of  any set $X\subset \mathbb R^n_\xi$, we introduce the concepts of characteristic filter of a $m$-pseudodifferential operator and of filter of weighted Sobolev regularity of $u\in \mathcal D'(\Omega)$. We can then give a result of microlocal propagation of singularities for solutions to (pseudo)differential equations. Some applications to linear partial differential  operators are  given in the last \S \ref{EX}.

\section{$\textup{m}$-pseudodifferential operators on $L^p$}\label{PDO}
{\bf Notations.}
For $\chi (\xi), \kappa (\xi)$ positive continuous functions of $\xi\in\mathbb R^n$ and $C, c$ positive constants, we set:
\begin{itemize}
\item $\chi(\xi)\asymp\kappa(\xi)$, if $c\leq \frac{\chi(\xi)}{\kappa(\xi)}\leq C$, for any $\xi \in\mathbb R^n $;
\item $\chi(\xi)\approx\chi(\eta)$ in a domain $D$ if  $c\leq \frac{\chi(\eta)}{\chi(\xi)}\leq C$ , for any $\xi, \eta\in D$;
\item $\langle \xi\rangle=\sqrt{1+\vert \xi\vert^2}$;
\item $K\subset\!\!\subset \Omega$ when $K$ is a compact subset of $\Omega$.
\end{itemize}
\begin{defn}\label{DEFWS1} A vector valued function $\Lambda(\xi)=\left( \lambda_1(\xi), \dots, \lambda_n(\xi) \right)$, $\xi \in\mathbb R^n$, with positive continuous components, is a weight vector if there exist positive constants $C, c$ such that for any $j=1,\dots,n$:
\begin{eqnarray}
  && c\langle \xi\rangle^c \le \lambda_j(\xi)\le C\langle \xi\rangle^C \, \text{(polynomial growth)};\label{PG}\\
&&\lambda_j(\xi)\ge c\vert\xi_j\vert\,\, \,\text{(M-condition)};\label{MA}\\
&& \lambda_j(\eta)\approx\lambda_j (\xi)\quad \text{when}\quad \sum_{k=1}^n\vert\xi_k-\eta_k\vert \lambda_k(\eta)^{-1}\le c\,\, \text{(slowly varying condition)}.\label{SW}
\end{eqnarray}
A positive real continuous function $m(\xi)$ is an admissible weight, associated to the weight vector $\Lambda(\xi)$, if for some positive constants $N, C ,c$
\begin{eqnarray}
&& m(\eta)\le C\, m(\xi) \left(  1+\vert \eta-\xi\vert\right)^N\,\, \text{(temperance)} ;\label{TEMP}\\
&& m(\eta)\approx m(\xi)\quad \text{when}\quad \sum_{k=1}^n\vert\xi_k-\eta_k\vert\lambda_k(\eta)^{-1}\le c.\label{SW1}
\end{eqnarray}
\end{defn}
Considering  $\eta=0$ and $\xi=0$ in \eqref{TEMP}, it follows that $c\langle \xi\rangle^{-N}\leq m(\xi)\leq C\langle\xi\rangle^N$.\\
It is trivial that any positive constant function on $\mathbb R^n$ is an admissible weight associated to any weight vector $\Lambda(\xi)$.\\
Consider $\tilde\Lambda(\xi)\asymp\Lambda(\xi)$, that is $\lambda_j(\xi)\asymp\tilde\lambda_j(\xi)$, $j=1,\dots,n$, then $\tilde\Lambda(\xi)$ is again a weight vector. Similarly $\tilde m(\xi)\asymp m(\xi)$ an admissible weight.
\begin{ex}\label{EXWS1}
\begin{enumerate}
\item
Consider the function $\langle \xi\rangle_M=\left(1+\sum_{j=1}^n \xi_j^{2m_j}\right)^{1/2}$, the so-called {\it quasi-homogeneous weight}, where $M=(m_1, \dots,m_n)\in \mathbb N^n $ and $\min\limits_{1\leq j\leq n} m_j=1$.
Then $\Lambda_M(\xi)=\left(\langle \xi\rangle_M ^{1/m_1},\dots, \langle\xi\rangle_M^{1/m_n} \right)$ is a weight vector.
\item Assume that the continuous function $\lambda(\xi)$ satisfies \eqref{PG} and the \emph{strong} slowly varying condition
\begin{equation}\label{SW2}
\lambda(\eta)\approx\lambda(\xi),\, \text{when for some}\, c,\mu>0\quad  \sum\limits_{j=1}^n\vert\eta_j-\xi_j\vert\left(\lambda(\eta)^{\frac1{\mu}}+\vert\eta_j\vert\right)^{-1}\leq c,
\end{equation}
then the vector $\Lambda(\xi):=\left(\lambda(\xi)^{\frac{1}{\mu}}+\vert\xi_1\vert,\dots,\lambda(\xi)^{\frac{1}{\mu}}+\vert\xi_n\vert\right)$ is a weight vector, see \cite[Proposition 1]{GM2} for the proof.
In such frame emphasis is given to the {\it multi-quasi-homogeneous} weight functions $\lambda_{\mathcal P}(\xi)=\left(\sum_{\alpha\in V(\mathcal P)}\xi^{2\alpha}\right)^{1/2}$, where $V(\mathcal P)$ is the set of the vertices of a {\em complete Newton polyhedron} $\mathcal P$ as introduced in \cite{GV}, see also \cite{BBR}; in this case, the value $\mu$ in \eqref{SW2} is  called {\em formal order} of $\mathcal P$. For some  details see Section \ref{EX} 

\item For any $s\in \mathbb R$, the functions $\langle \xi\rangle_M^s$, $\lambda(\xi)^s$ are admissible weights for the weight vectors respectively defined in 1. and 2.
\end{enumerate}
\end{ex}
\begin{rem}\label{REMWS1}
Consider the function $\lambda(\xi)$ such that $\lambda(\eta)\approx\lambda(\xi)$ when $\vert\eta-\xi\vert<c\lambda(\eta)^{\frac 1\mu}$,  for suitable  positive constants $\mu, c$. Since $\vert \xi-\eta\vert^\mu\leq c\lambda(\eta)$ implies $\lambda(\eta)\leq C\lambda(\xi)\leq C\lambda(\xi)\left(1+\vert \xi-\eta\vert\right)^{\mu}$, using moreover \eqref{PG}, we obtain that $\lambda(\xi)$ satisfies the temperance condition \eqref{TEMP} with constant $N=\mu$.
\end{rem}
\begin{prop}\label{PROWS1}
For $\Lambda(\xi)=\left(\lambda_1(\xi), \dots, \lambda_n(\xi)\right)$ weight vector, the following properties are satisfied:
\begin{itemize}
\item[i)]the function:
\begin{equation}\label{EQWS0}
\pi(\xi)=\min_{1\leq j\leq n}\lambda_j(\xi), \quad \xi\in \mathbb R^n
\end{equation}
is an admissible weight associated to $\Lambda(\xi)$ and it moreover satisfies \eqref{SW2}, with $\mu=1$;
\item[ii)] If $m,m'$ are admissible weights associated to the weight vector ${\Lambda(\xi)}$, then the same property is fulfilled by $mm'$ and $1/m$.
\end{itemize}
\end{prop}
\begin{proof}
In view of \eqref{MA} and \eqref{EQWS0}, the assumption $\sum\limits_{k=1}^n\vert\xi_k-\eta_k\vert\left(\pi(\eta)+\vert\eta_k\vert\right)^{-1}\leq c$ directly gives $\sum\limits_{k=1}^n\vert \xi_k-\eta_k\vert \lambda_k(\eta)^{-1}\leq \tilde c$, where $\tilde c>0$ depends increasingly on $c$. Then for suitably small $c$, we obtain from   the slowly varying condition \eqref{SW} and some $C>0$: $\frac{1}{C}\lambda_j(\xi)\leq\lambda_j(\eta)\leq C\lambda_j(\xi)$, for any $j=1,\dots,n$. It then follows:
$\frac{1}C\pi(\xi)=\frac{1}{C}\min\limits_j\lambda_j(\xi)\leq\pi(\eta)=\min\limits_j\lambda_j(\eta)\leq C \min\limits_j\lambda_j(\xi)=C \pi(\xi)$. Thus $\pi(\xi)$ satisfies \eqref{SW2} and in the same way we can prove that it fulfils \eqref{SW1}. Then by means of the previous remark, i) is proved.\\
$m(\eta)\leq Cm(\xi)\left( 1 + \vert\xi-\eta\vert\right)^N\iff 1/m(\xi)\leq C 1/m(\eta)\left(1+\vert\xi-\eta\vert\right)^N$; then interchanging $\xi$ and $\eta$ we immediately obtain that $1/m$ is temperate. It is then trivial to prove the remaining part of ii).
\end{proof}

\begin{defn}\label{DEFWS2}
For $\Omega$ open subset of $\mathbb R^n$, $\Lambda(\xi)$ weight vector and $m(\xi)$ admissible weight, the symbol class $S_{m,\Lambda}(\Omega)$ is given by all the functions $a(x,\xi)\in C^\infty(\Omega\times\mathbb R^n)$, such that, for any  $K\subset\!\! \subset\Omega$, $\alpha,\beta\in \mathbb Z^n_+$ and suitable $c_{\alpha, \beta, K}>0$:
\begin{equation}\label{EQWS1}
\sup_{x\in K}\vert \partial_\xi^\alpha\partial_x^\beta a(x,\xi)\vert\le c_{\alpha,\beta, K}\, m(\xi)\Lambda(\xi)^{-\alpha}, \quad  \xi\in \mathbb R^n
\end{equation}
where, with standard vectorial notation, $\Lambda(\xi)^\gamma=\prod_{k=1}^n \lambda_k(\xi)^{\gamma_k}$.
\end{defn}
\noindent
$S_{m, \Lambda}(\Omega)$ turns out to be a Fr\'echet space, with respect to the family of natural semi-norms defined as the best constants $c_{\alpha,\beta, K}$ involved in the estimates \eqref{EQWS1}.

Henceforth $\Lambda(\xi)$ will always be a weight vector and all  the admissible weights $m(\xi)$  will be referred to it.

\begin{rem}\label{REMWS2}
\begin{enumerate}
\item
Considering the constants $C,c$ in \eqref{PG} and $N$ in \eqref{TEMP}, the following relation with the usual H\"ormander \cite{HO1} symbol classes $S^m_{\rho,\delta}(\Omega)$, $0\leq\delta<\rho\leq 1$, is trivial:
\begin{equation} \label{EQWS2A}
S_{m,\Lambda}(\Omega)\subset S^N_{c, 0}(\Omega)\,.
\end{equation}
\item
If $m_1, m_2$ are admissible weights such that $m_1\leq Cm_2$, then $S_{m_1, \Lambda}(\Omega)\subset S_{m_2, \Lambda}(\Omega)$, with continuous imbedding. In particular  $S_{m_1,\Lambda}(\Omega)=S_{m_2,\Lambda}(\Omega)$, as long as $m_1\asymp m_2$.
\newline
When the admissible weight $m$ is an arbitrary positive constant function, the symbol class $S_{m,\Lambda}(\Omega)$ will be just denoted by $S_\Lambda(\Omega)$ and $a(x,\xi)\in S_\Lambda(\Omega)$ will be called a {\it zero order symbol}.
\item
Since for any $k\in\mathbb Z_+$ the admissible weight $\pi(\xi)^{-k}$ is less than $C^k\langle \xi \rangle^{-ck}$, for $m$ admissible weight we have
\begin{equation*}\label{EQWS2B}
\bigcap_{k\in\mathbb Z+}S_{m\pi^{-k},\Lambda}(\Omega)\subset \bigcap_{N\in\mathbb Z_+}S^{-N}_{1,0}(\Omega)=:S^{-\infty}(\Omega)\,.
\end{equation*}
On the other hand $a(x,\xi)\in S^{-\infty}(\Omega)$ means that, for any $\mu\in \mathbb R$, $ K\subset\!\!\subset \Omega$,
\begin{equation}\label{EQWS2C}
\sup_{x\in K}\vert\partial^\beta_x\partial^\alpha_\xi a(x,\xi)\vert\leq c_{\alpha,\beta, K}\langle\xi\rangle^{\mu-\vert\alpha\vert}.
\end{equation}
Recall now  that, for suitable $N,C>0$, $m(\xi)\geq \frac{1}{C}\langle\xi\rangle^{-N}$, $\pi(\xi)\leq C\langle\xi\rangle$ and $\lambda_j(\xi)\leq C\langle\xi\rangle^C$. Then setting,  for any fixed $\alpha \in\mathbb Z_+^n$ and arbitrary $k\in\mathbb Z_+$, $\mu=-N-k-(C-1)\vert\alpha\vert$ in \eqref{EQWS2C}, we obtain $\vert\partial^\beta_x\partial^\alpha_\xi a(x,\xi)\vert \leq c_{\alpha,\beta}m(\xi)\pi(\xi)^{-k}\Lambda(\xi)^{-\alpha}$, for suitable $c_{\alpha,\beta}$, that is $a(x,\xi)\in S_{m\pi^{-k},\Lambda}(\Omega)$ for any $k\in\mathbb Z_+$. Then
\begin{equation}\label{EQWS2D}
\bigcap_{k\in\mathbb Z_+}S_{m\pi^{-k},\Lambda}(\Omega)\equiv S^{-\infty}(\Omega).
\end{equation}
\item
Thanks to the relations with the H\"ormander symbol classes \eqref{EQWS2A}, we can define for $a(x,\xi)\in S_{m,\Lambda}(\Omega)$ the $m-$pseudodifferential operator $a(x,D)$ by means of \eqref{eqINT1}.\\
For any weight vector $\Lambda(\xi)$ and  $m(\xi)$ admissible weight, the operator $a(x,D)$ maps continuously $C^\infty_0(\Omega)$ to  $C^\infty(\Omega)$  and it extends to a bounded  linear operator from $\mathcal E'(\Omega)$ to $\mathcal D'(\Omega)$.
\item
By  $\textup{Op}S_{m,\Lambda}(\Omega)$ we denote the class of all the $m-$pseudodifferential operators with symbol in $S_{m,\Lambda}(\Omega)$, while $\widetilde{\textup{Op}}S_{m,\Lambda}(\Omega)$ is the class of {\em properly supported} $m-$pseudodifferential operators  which map $C^\infty_0(\Omega )$ to $C^\infty_0(\Omega)$, $C^\infty(\Omega)$ to $C^\infty(\Omega)$ and extend to  bounded linear operators into $\mathcal E'(\Omega)$ and $\mathcal D'(\Omega)$.
\item
When $a(x,D)$ is properly supported, for any fixed $\psi\in C^\infty_0(\Omega)$ we can find $\phi\in C^\infty_0(\Omega)$ such that $\psi a(x,D)u=\psi a(x,D)\phi u$, for any $u\in \mathcal D'(\Omega)$, see \cite[Prop. 3.4]{SH1}.
Moreover, for any $a(x, D)\in \textup{Op} S_{m,\Lambda}(\Omega)$, there exists $a'(x, D)\in\widetilde{\textup{Op}} S_{m,\Lambda}(\Omega)$ whose symbol satisfies $a'(x,\xi)\sim a(x,\xi)$, that is $a'(x,\xi)-a(x,\xi)\in S^{-\infty}(\Omega)$.
\end{enumerate}
\end{rem}
By means of the arguments in \cite[Proposition 1.1.6]{NR1}  and \cite{BE1}, jointly with Remark \ref{REMWS2} we obtain the following asymptotic expansion.
\begin{prop}\label{PROPWS2}
For any sequence of symbols $a_k(x,\xi)\in S_{m\pi^{-k},\Lambda}(\Omega)$, $k\in\mathbb Z_+$ there exists $a(x,\xi)\in S_{m,\Lambda}(\Omega)$ such that for every integer $N\ge 1$:
\begin{equation}\label{EQWS2F}
a(x,\xi)-\sum_{k<N}a_k(x,\xi)\in S_{m\pi^{-N},\Lambda}(\Omega)\,.
\end{equation}
Moreover $a(x,\xi)$ is uniquely defined modulo  symbols in $S^{-\infty}(\Omega)$.
\end{prop}
\noindent
We write
\begin{equation}\label{EQWS2G}
a(x,\xi)\sim\sum_{k=0}^\infty a_k(x,\xi),
\end{equation}
if for every $N\geq 1$ \eqref{EQWS2F} holds.

The adjoint  $a(x,D)^\ast$ of the $m-$pseudodifferential operator $a(x,D)\in \,{\textup{Op}} S_{m, \Lambda}(\Omega)$ is defined by the identity:
\begin{equation}\label{ADJOINT}
(a(x,D) \varphi,\psi)=(\varphi, a(x,D)^\ast \psi), \quad \text{for any}\quad \varphi,\psi\in C^\infty_0(\Omega),
\end{equation}
where $(f,g)=\int f\bar g$.
\begin{prop}[Asymptotic expansion]\label{PROMM1}
Consider $a(x,D)\in\textup{Op}\,S_{m, \Lambda}(\Omega)$ and $b(x,D)\in\widetilde{\textup{Op}}\,S_{m', \Lambda}(\Omega)$, where $m(\xi), m'(\xi)$ are admissible weights, both associated to the same weight vector $\Lambda(\xi)$. Then we have:
\begin{itemize}
 \item [i)]$a(x, D)^\ast\in\,\textup{Op}\,S_{m, \Lambda}(\Omega)$  and $a(x, D)^\ast=a^\ast(x, D)$, where  $a^\ast(x,\xi)\in S_{m,\Lambda}(\Omega)$ satisfies the following aymptotic expansion:
\begin{equation}\label{EQMM2}
a^\ast(x,\xi)\sim\sum_{\alpha}\frac{1}{\alpha !}\partial^\alpha_\xi D^\alpha_x\bar a(x,\xi))\,,\qquad D^\alpha:=(-i)^{\vert\alpha\vert}\partial^\alpha\,.
\end{equation}
Moreover $a(x, D)^\ast\in\, \widetilde{\textup{Op}}\, S_{m,\Lambda}(\Omega)$ if $a(x,D)$ is assumed properly supported.
\item [ii)] $b(x,D)a(x,D)=c(x,D)\in\,\textup{Op}\,S_{mm',\Lambda}(\Omega) $  and
\begin{equation}\label{EQWS10}
c(x,\xi)\sim\sum_{\alpha}\frac{1}{\alpha !}\partial^\alpha_\xi b(x,\xi)D^\alpha_x a(x,\xi).
\end{equation}
\end{itemize}
\end{prop}
\noindent The proof directly follows from the arguments in \cite[\S 1.2.2]{NR1}, and \cite [\S I.3] {TRE1}.

Considering $a(x,\xi)\in S_{m,\Lambda}(\Omega)$ and using \eqref{EQWS1}, \eqref{TEMP}, \eqref{PG}, \eqref{MA}, it immediately follows that for any $\alpha,\gamma \in\mathbb Z_+^n$, $K\subset\!\!\subset\Omega$,
    \begin{equation*}
    \sup\limits_{x\in K}\vert \xi^\gamma \partial^{\alpha+\gamma}_\xi a(x, \xi)\vert \leq M_{\alpha,\gamma,K} \langle\xi\rangle^{N-c|\alpha|}
 \end{equation*}
with some positive constants $M_{\alpha,\gamma,K}$, $N, c$. Then $S_{m,\Lambda}(\Omega)\subset M^N_{c, 0}(\Omega)$. Here $M^r_{\rho,0}(\Omega)$, $0<\rho\leq 1$, are the symbol classes defined in Taylor \cite  {TA1} given by all the symbols $a(x,\xi)\in S^r_{\rho,0}(\Omega)$ such that for any $\gamma\in\{0,1\}^n$, $\xi^\gamma\partial^\gamma_\xi a(x,\xi)\in S^r_{\rho, 0}(\Omega)$.
Then applying the arguments in \cite[Ch. XI, Prop. 4.5]{TA1}, see also \cite[Theorem 4.1]{GM1}, the next result immediately follows:

\begin{thm}[continuity]\label{CONT}
If $a(x,\xi)\in S_\Lambda(\Omega)$, then, for any $1<p<\infty$:
\begin{equation*}
a(x,D): L^p_{\rm comp}(\Omega)\mapsto L^p_{\rm loc }(\Omega) \quad \text{continuously}.
\end{equation*}
If $a(x,D)$  is assumed to be properly supported, then it is bounded both as operator into $L^p_{\rm comp}(\Omega)$ and into $L^p_{\rm loc}(\Omega)$.
\end{thm}
Let $ p(\xi)$ be a smooth function  satisfying for every $\alpha \in \mathbb Z^n_+$ the estimate
\begin{equation}\label{GLOBAL}
\vert\partial^\alpha p(\xi)\vert < c_{\alpha}m(\xi) \Lambda(\xi)^{-\alpha}, \quad \text{for some }\,\, c_\alpha>0.
\end{equation}
Then, by means of  \eqref{eqINT1}, $p(D)u:= (2\pi)^{-n}\int e^{i x\cdot \xi}p(\xi)\hat u(\xi)=\mathcal F^{-1}(p\hat u)$  defines a  Fourier multiplier which maps the Schwartz space $S(\mathbb R^n)$ to itself, and extends to a bounded map $S'(\mathbb R^n)\mapsto S'(\mathbb R^n)$.\\
It may be proved by means of  technical arguments, see  \cite[\S 2, Prop. 3]{GM2}, that for any admissible weight $m(\xi)$ there exists a smooth equivalent weight $\tilde m(\xi)$ whose derivatives satisfy the estimates  in \eqref{GLOBAL}.\\
Identifying now $m(\xi)$ and $\tilde m(\xi)$, we can define for $1<p<\infty$ the weighted Sobolev space:
\begin{equation}\label{WSS}
H^p_m:=\left\{u\in \mathcal S'(\mathbb R^n),\,\, \text{such that}\,\,\, m(D)u\in L^p(\mathbb R^n)\right\}.
\end{equation}
$H^p_m$ may be equipped in natural way by the norm $\Vert u \Vert_{p,m}:=\Vert m(D)u\Vert_{L^p}$ and it  realizes to be a Banach space (Hilbert space in the case $p=2$, with inner product $(u,v)_m=\left( m(D)u, m(D)v \right)_{L^2}$ ).\\
Using standard arguments it can be proved that  $ \mathcal S(\mathbb R^n)\subset H^p_m\subset \mathcal S'(\mathbb R^n)$, with continuous embeddings and moreover $\mathcal S(\mathbb R^n)$ is dense in $H^p_{m}$, $1<p<\infty$.\\
For any open subset $\Omega \subset \mathbb R^n$ the following local  spaces may be introduced:
\begin{eqnarray}
&&H^p_{m, {\rm loc}}(\Omega)=\left\{u\in \mathcal D'(\Omega)\,\,\text{such that, for any}\, \varphi\in C^\infty_0(\Omega),\, \varphi u\in H^p_m  \right\}.\label{EQSS4}\\
&&H^p_{m,{\rm comp}}(\Omega)=\bigcup_{K\subset\!\!\subset \Omega}H^p_m(K), \label{EQSS3}
\end{eqnarray}
where $H^p_m(K)$ is the closed subspace of $H^p_m$, consisting of the distributions supported in the compact set $K$.\\
$H^p_{m,{\rm loc}}(\Omega)$ equipped with the family of seminorms $p_{\psi} (\cdot):=\Vert \psi \cdot\Vert_{m,p}=\Vert m(D)\psi\cdot\Vert_{L^p}$, $\psi\in C^\infty_0(\Omega)$ arbitrary,  is a Fr\'echet space.\\
$H^p_{m, \textup{comp}}(\Omega)$ is provided with the inductive limit topology of the spaces $H^p_m(K)$, for $K$ ranging on the collection of all compact subsets of $\Omega$.\\
For any $1<p<\infty$ we have the following embeddings with dense inclusion:
\begin{equation*}\label{DENSITY}
C^\infty_0(\Omega)\hookrightarrow H^p_{m, \textup{comp}}(\Omega)\hookrightarrow H^p_{m,\textup{loc}}(\Omega)\hookrightarrow \mathcal D'(\Omega).
\end{equation*}
Considering now $m'(\xi)\geq cm(\xi)$, thanks to Theorem \ref{CONT}, we have for any $\psi\in C^\infty_0(\Omega)$ and $u\in C^\infty(\Omega)$, 
\begin{equation*}\Vert m(D)(\psi u)\Vert_{L^p}= \Vert \frac{m(D)}{m'(D)}m'(D)(\psi u)\Vert_{L^p}\leq C\Vert m'(D)(\psi u)\Vert_{L^p}.
\end{equation*} 
The next local Sobolev embedding then immediately follows:
\begin{equation}\label{SOBEMB}
H^p_{m',{\rm loc}}(\Omega)\hookrightarrow H^p_{m,{\rm loc}}(\Omega), \quad \text {when}\,\, m'(\xi)\geq c\, m(\xi), \,\, \text{for some}\,\, c>0.
\end{equation}
\begin{prop}\label{SOBOLEVCONT}
Assume $m, m'$ admissible weights,  $a(x,\xi)\in S_{m',\Lambda}(\Omega)$ and  $p\in]1,\infty[$. Then $a(x,D)$ extends to a bounded linear operator:
\begin{eqnarray}
&&a(x,D):H^p_{m, {\rm comp}}(\Omega)\mapsto H^p_{m/m', {\rm loc}} (\Omega).\label{EQSS6}
\end{eqnarray}
If moreover $a(x,D)$ is a properly supported operator  then the following maps are continuous:
\begin{eqnarray}
&&a(x,D):H^p_{m, {\rm comp}}(\Omega)\mapsto H^p_{m/m',{\rm comp}}(\Omega);\label{EQSS7}\\
&&a(x,D):H^p_{m,{\rm loc}} (\Omega)\mapsto H^p_{m/m',{\rm loc}}(\Omega).\label{EQSS8}
\end{eqnarray}
\end{prop}
\begin{proof} Thanks to the topology of $H^p_{m,\textup{comp}}(\Omega)$, we need that for every $K\subset\!\!\subset \Omega$ $a(x,D)$ maps continuously $H^p_{m}(K)$ to $H^p_{m/m', \textup{loc}}(\Omega)$. To such purpose consider  $\phi\in C^\infty_0(\Omega)$, such that $\phi(\xi)=1$ in $K$, then for any $\psi\in C^\infty_0(\Omega)$ we have:
\begin{equation}\label{EQWS10BIS}
\begin{array}{l}
\Vert \psi a(x,D)u\Vert_{H^p_{m/m'}}=\Vert \psi a(x,D)(\phi u)\Vert_{H^p_{m/m'}}=\\
\Vert \frac{m}{m'}(D)\psi a(x,D)\phi \frac{m}{m}(D) u\Vert_{L^p}=\Vert \frac{m}{m'}(D)\psi a(x,D)\phi \frac{1}{m}(D)m(D) u\Vert_{L^p}\,.
\end{array}
\end{equation}
Since $\frac{m}{m'}(D)\psi a(x,D)\phi \frac{1}{m}(D)$ admits symbol in $S_{\Lambda}(\Omega)$ and moreover it extends to an operator from  $\mathcal S'(\mathbb R^n)$ to itself, we obtain that $\Vert \psi a(x,D)u\Vert_{H^p_{m/m'}}\leq C \Vert u\Vert_{H_m^p}$, which shows \eqref{EQSS6}.\\
Considering now $a(x,D)$ properly supported,  the operator $a(x,D)\phi\frac{1}{m}(D)m(D)$ maps $\mathcal E'(\Omega)$ into itself, thus \eqref{EQSS7} directly follows from \eqref{EQWS10BIS}.\\
Since $a(x,D)$ is properly supported, for any $\psi\in C^\infty_0(\Omega)$ we can find another test function $\phi\in C^\infty_0(\Omega)$, such that $\psi a(x,D)u=\psi a(x,D) \phi u$, thus \eqref{EQSS8} directly follows by the calculus in \eqref{EQWS10BIS}.
\end{proof}
\begin{defn}\label{ELLIPT} We say that a symbol $a(x,\xi)\in S_{m, \Lambda}(\Omega)$, or equivalentely the operator $a(x,D)$, is $m$-elliptic if for every $K\subset\!\!\subset\Omega$, two positive constants $c_K, C_K$ exist such that
\begin{equation}\label{eqELLIPT1}
\vert a(x,\xi)\vert>c_k m(\xi), \quad \text {for any}\quad x\in K\quad\text{and}\quad \vert \xi\vert> C_K.
\end{equation}
\end{defn}

\begin{prop}[parametrix]\label{PROPSS2}
Let $a(x,\xi)\in S_{m,\Lambda}(\Omega)$ be a $m-$elliptic symbol. Then a properly supported  operator $b(x,D)\in \widetilde{\textup{Op}}S_{1/m, \Lambda}(\Omega)$ exists such that:
\begin{equation}\label{EQSS12}
b(x,D)a(x,D)=\textup{Id}+\rho(x,D),
\end{equation}
where $\rho(x,\xi)\in S^{-\infty}(\Omega)$ and ${\rm Id}$ denotes the identity operator.
\end{prop}
See \cite[ Theorem 1.3.6]{NR1} for the proof.


\section{Minimal and maximal $m$-pseudodifferential operators}\label{MM}
As introduction to this section let us recall some basic facts about duality in a generical  locally convex topological vector space $X$, for details and proofs the reader can refer to \cite[Ch. VII]{YO1}\\
We say strong dual of $X$ the space $X'$  of linear continuous functions from $X$ to the complex field $\mathbb C$ and write $\langle x, x' \rangle=x'(x)$, $x\in X$, $x'\in X'$, endowed with the strong topology defined by the seminorms $p_B(x')=\sup_{x\in B}\vert \langle x,x'\rangle\vert$, for $B$ arbitrary bounded subset of $X$.\\
For $X, Y$ locally convex topological vector spaces, take a linear operator $T:X\mapsto Y$, defined in a linear subspace  $D(T)\subset X$  , and consider a couple $(x', y')\in X'\times Y'$ which satisfies:
\begin{equation}\label{DUAL1}
\langle Tx, y'\rangle=\langle x, x'\rangle, \quad\text{for any}\,\,\, x\in D(T).
\end{equation}
Then $x'$ is uniquely determined by $y'$ if and only if $D(T)$ is dense in $X$, \cite[Ch. VII, Theorem 1]{YO1}.\\
Thus for any linear operator $T:X\mapsto Y$, with domain $D(T)$ dense in $X$, we can define  the \emph{dual operator} $T': Y'\mapsto X$, with domain $D(T')$ given by all the $y'\in Y'$ such that \eqref{DUAL1} is satisfied for some $x'\in X'$. $T'$ is  univocally defined by setting $T'y'=x'$. It moreover follows directly from \eqref{DUAL1}:
\begin{equation}\label{DUAL2}
\langle Tx, y'\rangle=\langle x, T'y'\rangle,\quad \text{for any}\quad x\in D(T), \, y'\in D(T').
\end{equation}
Recall now  that a  linear operator $T:\mathcal D(T)\mapsto Y$, $D(T)$ linear subspace of  $X$, is said to be closed when its graph $G(T)=\left\{(x, Tx),; x\in D(T)\right\}$ is a  closed linear subspace of $X\times Y$. $T$ is closable if the closure of $G(T)$ in $X\times Y$ is the graph of a linear operator mapping $X$ to $Y$.\\
If $X, Y$ are Fr\'echet spaces, or more generally   quasi normed spaces, then $T$ is closed (closable) iff, for every sequence $\{x_n\} \subset D(T)$,
\begin{eqnarray}
 \lim_{n\to\infty} x_n=x, \lim_{n\to\infty}Tx_n=y,\, x\in X, y\in Y\,\, \Rightarrow x\in D(T),\,\,Tx=y\label{EQMM1}\\
\left( \lim_{n\to\infty} x_n=0\, \,\text{and}\,\, \lim_{n\to\infty}Tx_n=y\,\,\Rightarrow\,\, y=0\right)\label{EQMM2BIS}.
\end{eqnarray}
Consider a closable linear operator $T:X\mapsto Y$, $X, Y$ Fr\'echet spaces; we can define in a unique way the linear closed operator  $T_0$ as follows: $x\in D(T_0)$ iff there exists  a sequence $\{x_n\}\subset D(T)$ such that $\lim_{n\to \infty }x_n=x$ and $\lim_{n\to \infty}T x_n=y\in Y$ exists, we set then $T_0 x=y$. $T_0$ is called the {\it smallest closed extension of T} in $X$. For more details the reader can see \cite[Ch II, \S 6]{YO1}.\\
Lacking for an adequate reference in Fr\'echet spaces, we prove the following
\begin{prop}\label{PROMM4}
Consider $X, Y$  Fr\'echet spaces, $X', Y'$ their strong dual spaces, $T:X\mapsto Y$ linear operator with domain $D(T)$ dense in $X$ and $T'$ the dual operator of $T$. Then
\begin{itemize}
\item[i)] $T'$ is a closed linear operator from $D(T')\subset Y'$ to $X'$.
\item[ii)] If $W$ is an extension of $T$, then $T'$ is an extension of $W'$.
\end{itemize}
\end{prop}
\begin{proof}
Assume that $y'_k\to y'$ in $Y'$ and $T'y'_k\to x'$ in $X'$. By means of \eqref{DUAL2} we have, for  $k=1,2,\dots$, $\langle Tx, y'_k\rangle=\langle x, T'y'_k  \rangle$, for any $x\in D(T)$. Letting now $k\to \infty$ we obtain,  for any $x\in D(T)$, $\langle Tx, y'\rangle=\langle x,x'\rangle$. It follows that $y'\in D(T')$ and $x'=T'y'$, which proves i).\\
Take now $y'\in D(W')$ and assume that, for some $x'\in X'$ we have $\langle Wx, y' \rangle=\langle x, x'\rangle$, for any $x\in D(W)$. Since $W$ is an extension of $T$, it follows that $\langle Tx, y' \rangle=\langle x, x'\rangle$, again for any $x\in D(T)$. Thus $y'\in D(T')$ and $T' y'=x'=W' y'$ and the proof is concluded.
 \end{proof}

Hereafter, assuming that $m(\xi)\geq c>0$,  we consider the properly supported $m-$pseudodifferential operators noted by $T_a=a(x,D)\in\widetilde{\textup{Op}}S_{m,\Lambda}(\Omega)$ as linear operators from $L^p_{\rm loc}(\Omega)$ to itself with dense domain $D(T_a)=C^\infty_0(\Omega)$.
\begin{prop}\label{PROMM2}  $T_a$ is a closable operator on $L^p_{\rm loc}(\Omega)$, for any $a(x,\xi)\in S_{m,\Lambda}(\Omega)$, $m(\xi)\geq c>0$ and $1<p<\infty$.
\end{prop}
\begin{proof}
Recall that $\widetilde{\textup{Op}}S_{m,\Lambda}(\Omega)$ is continuous into $C^\infty_0(\Omega)$  and consider a sequence $f_k\in C^\infty_0(\Omega)$ such that $f_k\stackrel{k\to\infty}{\rightarrow} 0$ and  $T_a f_k\stackrel{k\to\infty}{\rightarrow} g$, both in $L^p_{\rm loc}(\Omega)$. For fixed $k$,  taking any $\psi\in C^\infty_0(\Omega)$, we have:
\begin{equation}\label{EQMM3}
\vert\left(T_a f_k, \psi \right)\vert=\vert\left(f_k, T^\ast_a \psi\right)\vert\leq \Vert f_k\Vert_{L^p}\Vert T^\ast_a  \psi\Vert_{L^{p'}}, \quad \frac{1}{p}+\frac{1}{p'}=1.
\end{equation}
Notice  that $\Vert f_k\Vert_{L^p}\rightarrow 0$, $\left( T_af_k, \psi\right)\rightarrow(g,\psi)$ in $\mathbb C$ and $\psi$ ranges all over $C^\infty_0(\Omega)$. Thus we can end that $g(x)=0$ almost everywhere, then $g=0$ in   $L^p_{\rm loc}(\Omega)$. Since $C^\infty_0(\Omega)$ is dense in $L^p_{\rm loc}(\Omega)$, the proof is concluded.
\end{proof}
We can then consider the  smallest closed extension $T_{a, 0}$ of $T_a$ and call it \textit{minimal extension of} $T_a$ \\
Define now the operator $T_{a,1}$ as follows.
\begin{defn}\label{DEFMAXIMAL} For  \,$T_a\in \,\widetilde{\textup{Op}}\, S_{m,\Lambda}(\Omega)$ , $u\in L^p_{\rm loc}(\Omega)$ belongs to the domain $D(T_{a,1})$ if there exists $f\in L^p_{\rm loc}(\Omega)$ such that
\begin{equation}\label{MAXIMAL1}
(u, T^\ast_a \varphi)=(f, \varphi), \quad \text{for any}\quad \varphi\in C^\infty_0(\Omega).
\end{equation}
We set then $T_{a,1} u=f$.
\end{defn}
$T_{a,1}$ is called \textit{maximal extension of} $T_a$.
Clearly $C^\infty_0(\Omega)\subset D(T_{a,1})$ and $T_{a,1}$ is linear.
\begin{prop}\label{PROMAXIMAL1}
For any $T_a\in \, \widetilde {\textup{Op}}\, S_{m,\Lambda}$,
\begin{itemize}
\item[i)] $T_{a,1}$ coincides with $T_a$ in distribution sense;
\item[ii)] $T_{a,1}$ is a closed linear operator in $L^p_{\rm loc}(\Omega)$.
\end{itemize}
\end{prop}
\begin{proof}
Consider both $u$ and $T_{a,1}u$ as distributions in $\mathcal D'(\Omega)$.  Definition \ref{DEFMAXIMAL} gives $\langle T_{a,1} u, \bar\varphi\rangle=\langle u, \overline{T_a^\ast \varphi} \rangle$ for any  $\varphi\in C^\infty_0(\Omega)$. On the other hand using \eqref{ADJOINT} we have $\langle T_a u, \bar\varphi \rangle=\langle u, \overline{T^\ast_a \varphi} \rangle$. Since $\bar\varphi$ ranges all over $C^\infty_0(\Omega)$, we conclude that $T_a u=T_{a,1}u$ in $\mathcal D'(\Omega)$.\\
Consider   now a sequence $\{u_k\}\subset D(T_{a,1})$ such that $u_k\to u$ and $T_{a,1}u_k\to f$, both in $L^p_{\rm loc}(\Omega)$. Definition \ref{DEFMAXIMAL} assures that, for any $\varphi \in C^\infty_0(\Omega)$ and $k=0,1\dots$, $(u_k, T^\ast_a \varphi)=(T_{a,1}u_k, \varphi)$. Thus for $k\to\infty$ we have $(u, T^\ast_a \varphi)=(f, \varphi)$, that is $u\in D(T_{a,1})$ and $f=T_{a,1}u$. The proof is then concluded.
\end{proof}
For the dual operator $T'_{a,1}$ defined by \eqref{DUAL2} we have the following properties.
\begin{prop}\label{PRODUAL}
Consider $a(x,\xi)\in S_{m,\Lambda}(\Omega)$, with $m(\xi)\geq c>0$, $\xi\in\mathbb R^n$. Then
\begin{itemize}
\item[i.] $C^\infty_0(\Omega)\subset D(T'_{a,1})$;
\item[ii.] $T_{a,1}$ is an extension of $T_{a,0}$;
\item[iii.] $T_{a,1}$ is the largest closed extension of $T_a$ having $C^\infty_0(\Omega)$ contained in the domain of its adjoint $T'_{a,1}$.
\end{itemize}
\end{prop}
\begin{proof}
We follow here the same lines as in the proof of \cite[Propositions 12.6-12.9]{wong_pdo}.
\newline
{\it Statement i.} Recall that $D(T'_{a,1})$ is defined to be the vector space of linear functionals $\ell\in\left(L^p_{\rm loc}(\Omega)\right)^\prime$ such that
\begin{equation}\label{lf1}
\begin{array}{ll}
D(T_{a,1})\mapsto\mathbb C\\
\qquad u\mapsto\langle\ell, T_{a,1}u\rangle
\end{array}
\end{equation}
may be extended to a (unique) linear functional $T'_{a,1}\ell\in\left(L^p_{\rm loc}(\Omega)\right)^\prime$. This means that if for some $\lambda\in \left(L^p_{\textup{loc}}(\Omega)\right)'$, $\langle T_{a,1}u,\ell\rangle=\langle u,\lambda\rangle$, then $\lambda=T'_{a,1}\ell$; see \eqref{DUAL1}. Let $\varphi\in C^\infty_0(\Omega)$ be arbitrarily given.  Using Definition \ref{DEFMAXIMAL}, we obtain for  $u\in D(T_{a,1})$:
\begin{equation}\label{lf1a}
\langle T_{a,1}u, \varphi\rangle=(T_{a,1}u,\overline\varphi)=(u,T^\ast_{a}\overline\varphi)=\langle u,\lambda\rangle\,,
\end{equation}
with $\lambda=\overline{T^\ast_{a }\bar\varphi}$.
On the other hand, H\"{o}lder's inequality gives
\begin{equation}\label{holder}
\vert\langle u,\lambda\rangle\vert=|(u,T^\ast_{a}\overline\varphi)|\le\Vert T^\ast_a\overline\varphi\Vert_{L^{p'}(K)}\Vert u\Vert_{L^p(K)}\,,
\end{equation}
where $\frac1{p}+\frac1{p'}=1$ and $K$ is a compact subset of $\Omega$ containing the support of $T^\ast_a\overline\varphi$ (recall that $T^\ast_a$ is a properly supported operator and  see Proposition \ref{PROMM1}). The estimate \eqref{holder} shows that the linear functional
\begin{equation}\label{lf2}
D(T_{a,1})\ni u\mapsto\langle T_{a,1}u, \varphi\rangle
\end{equation}
extends to a unique continuous linear functional in $\left(L^p_{\rm loc}(\Omega)\right)^\prime$; then $\varphi\in D(T'_{a,1})$, in view of \eqref{lf1a}.
\newline
{\it Statement ii.} For arbitrary $u\in D(T_{a,0})$, there exist a sequence $\{\varphi_k\}$ in $C^\infty_0(\Omega)$ and a function $f\in L^p_{\rm loc}(\Omega)$ such that
\begin{equation}\label{conv1}
\varphi_k\to u\quad\mbox{and}\quad T_a\varphi_k\to f\quad\mbox{in}\,\,L^p_{\rm loc}(\Omega)\,,\,\,\mbox{as}\,\,k\to +\infty\,,
\end{equation}
and we set $T_{a,0}u=f$.
\newline
For every  $\psi\in C^\infty_0(\Omega)$, from \eqref{conv1} it follows that
\begin{equation}\label{conv2}
(\varphi_k,T^\ast_a\psi)\to (u,T^\ast_a\psi)\quad\mbox{and}\quad(T_a\varphi_k,\psi)\to(f,\psi)\,,\quad\mbox{as}\,\,k\to +\infty\,.
\end{equation}
Since $(T_a\varphi_k,\psi)=(\varphi_k,T^\ast_a\psi)$ for all  $k$, then $(f,\psi)=(u,T_a^\ast\psi)$, thus $u\in D(T_{a,1})$ and $T_{a,1}u=f$.
\newline
{\it Statement iii.} Let $B:D(B)\subset L^p_{\rm loc}(\Omega)\mapsto L^p_{\rm loc}(\Omega)$ be a closed extension of $T_a$ such that $C^\infty_0(\Omega)$ is included in $D(B')$. We need to prove that $T_{a,1}$ is an extension of $B$. In view of Proposition \ref{PROMM4}, we already know that $T'_{a}$ extends $B'$; let us consider $u\in D(B)$ and $\psi\in C^\infty_0(\Omega)$; since $C^\infty_0(\Omega)\subset D(B')$, we get
\begin{equation}\label{dualita1}
\langle u, T'_a\psi\rangle=\langle u, B'\psi\rangle=\langle Bu, \psi\rangle\,.
\end{equation}
On the other hand, using the identity $T'_a\psi=\overline{T^\ast_a\overline\psi}$, from \eqref{dualita1} we get
\begin{equation*}
(u,T^\ast_a\overline\psi)=\langle u,\overline{T^\ast_a\overline\psi}\rangle=\langle\psi, Bu\rangle=(Bu, \overline\psi)\,.
\end{equation*}
Since $\psi$ is arbitrary in $C^\infty_0(\Omega)$, the relation above shows that $u\in D(T_{a,1})$  and $T_{a,1}u=Bu$.
\end{proof}
\begin{rem}
Thanks to the  Propositions \ref{PROMM4} and \ref{PRODUAL}  $T'_{a,0}$ is an extension of $T'_{a,1}$, and its domain $D(T'_{a,0})$ contains $C^\infty_0(\Omega)$.
\end{rem}
The following is an extension of the Agmon-Douglis-Nirenberg inequality \cite{ADN}.
\begin{lemma}\label{lemma_ell}
Let the weight function $m=m(\xi)$ satisfy $m(\xi)\ge c>0$, for all $\xi\in\mathbb R^n$, and the operator $T_a\in \widetilde{\textup{Op}}\,S_{m,\Lambda}(\Omega)$ be $m-$elliptic. Then for every $\chi\in C^\infty_0(\Omega)$ there exist $\widetilde\chi\in C^\infty_0(\Omega)$ and $C_\chi>0$ such that the following holds
\begin{equation}\label{ell_est}
\Vert\chi u\Vert_{H^p_m}\le C_\chi\left\{\Vert \widetilde\chi T_au\Vert_{L^p}+\Vert\widetilde\chi u\Vert_{L^p}\right\}\,,\quad\forall\,u\in H^p_{m, {\rm loc}}(\Omega)\,\footnote{Notice that $H^p_{m,{\rm loc}}(\Omega)\hookrightarrow L^p_{\rm loc}(\Omega)$, in view of $m(\xi)\ge c>0$ cf. \eqref{SOBEMB}.}.
\end{equation}
\end{lemma}
\begin{proof}
Thanks to Proposition \ref{PROPSS2}, a properly supported operator $T_b$ exists, such that $b(x,\xi)\in S_{\frac 1m, \Lambda}(\Omega)$ and $T_bT_a={\rm Id}+T_\rho$, with $\rho(x,\xi)\in S^{-\infty}(\Omega)$. Since $T_b$ and $T_\rho$ are properly supported, for an arbitrary $\chi\in C^\infty_0(\Omega)$ another function $\widetilde\chi\in C^\infty_0(\Omega)$ may be found such that
\begin{equation}\label{id1}
\chi u=\chi T_b\widetilde\chi T_a u-\chi T_\rho\widetilde\chi u\,,\quad\forall\,u\in H^p_{m,{\rm loc}}(\Omega)\,.
\end{equation}
Since, by Proposition \ref{SOBOLEVCONT}, $\chi T_b=T_{\chi b}$ and $\chi T_\rho=T_{\chi\rho}$ extend to continuous linear operators from $H^p_{m,\textup{loc}}(\Omega)$ into $L^p_{\textup{loc}}(\Omega)$ and moreover $T_\rho$ is a regularizing operator, from \eqref{id1} we get, for a suitable positive constant $C_\chi$,
\begin{equation}
\begin{array}{ll}
\Vert\chi u\Vert_{H^p_m}\le\left\{\Vert T_{\chi b}\widetilde\chi T_a u\Vert_{H^p_m}+\Vert T_{\chi\rho}\widetilde\chi u\Vert_{H^p_m}\right\}\\
\\
\quad \le C_\chi\left\{\Vert \widetilde\chi T_au\Vert_{L^p}+\Vert\widetilde\chi u\Vert_{L^p}\right\}\,,
\end{array}
\end{equation}
that gives the desired estimate \eqref{ell_est}.
\end{proof}
\begin{thm}\label{ADN}
Consider an admissible weight  $m(\xi)\ge c>0$, for all $\xi\in\mathbb R^n$, and let the operator $T_a\in \widetilde{\textup{Op}}\,S_{m,\Lambda}(\Omega)$ be $m-$elliptic. Then
\begin{itemize}
\item[i)] $D(T_{a, 0})=H^p_{m,{\rm loc}}(\Omega)$;
\item[ii)] $T_{a, 0}=T_{a,1}$.
\end{itemize}
\end{thm}
\begin{proof}
{\it Statement i)} Let $u\in H^p_{m,{\rm loc}}(\Omega)$. Since $C^\infty_0(\Omega)$ is dense in $H^p_{m,{\rm loc}}(\Omega)$, we may find a sequence $\left\{\varphi_k\right\}$ of functions in $C^\infty_0(\Omega)$ such that $\varphi_k\to u$ in $H^p_{m,{\rm loc}}(\Omega)$.  $H^p_{m,{\rm loc}}(\Omega)\hookrightarrow L^p_{\rm loc}(\Omega)$ since $m(\xi)\ge c>0$, then using Proposition \ref{SOBOLEVCONT} we  obtain that
\begin{equation}\label{conv3}
\varphi_k\to u\quad\mbox{and}\quad T_a\varphi_k\to T_au\,,\,\,\mbox{as}\,\,k\to\infty,\quad\mbox{in}\,\,L^p_{\rm loc}(\Omega)\,.
\end{equation}
This shows that $u\in D(T_{a,0})$, then $H^p_{m,{\rm loc}}(\Omega)\subseteq D(T_{a,0})$.
\newline
Conversely, assume $u\in D(T_{a,0})$; by definition of $T_{a,0}$ there exists again a sequence $\left\{\varphi_k\right\}$ in $C^\infty_0(\Omega)$ such that
\begin{equation}\label{conv4}
\varphi_k\to u\quad\mbox{and}\quad T_a\varphi_k\to f\,,\,\,\mbox{as}\,\,k\to\infty,\quad\mbox{in}\,\,L^p_{\rm loc}(\Omega)\,,
\end{equation}
for some $f\in L^p_{\rm loc}(\Omega)$. Hence $\left\{\varphi_k\right\}$ and $\left\{T_a\varphi_k\right\}$ are Cauchy sequences in $L^p_{\rm loc}(\Omega)$ and, thanks to the estimate \eqref{ell_est} (where $u$ is replaced by $\varphi_k-\varphi_h$), $\left\{\varphi_k\right\}$ is a Cauchy sequence in $H^p_{m,{\rm loc}}(\Omega)$,  then $\varphi_k\to v$, for some $v\in H^p_{m,{\rm loc}}(\Omega)$, as $k\to \infty$. But, since the sequence  $\left\{\varphi_k\right\}$ is also convergent to $u$ in $L^p_{\rm loc}(\Omega)$, in view of \eqref{SOBEMB}, we have that $u=v\in H^p_{m,{\rm loc}}(\Omega)$. This shows that $D(T_{a,0})\subseteq H^p_{m,{\rm loc}}(\Omega)$.
\newline
{\it Statement ii)} $T_{a,0}$ is the smallest closed extension of $T_a:C^\infty_0(\Omega)\mapsto L^p_{\rm loc}(\Omega)$, then, thanks to Proposition \ref{PRODUAL} and the first part of theorem, it suffices to prove that $D(T_{a,1})\subseteq H^p_{m,{\rm loc}}(\Omega)$. Let $u\in D(T_{a,1})$,  arguing on the ellipticity of the symbol $a(x,\xi)$, as in the proof of Proposition \ref{PRODUAL}, we find that
\begin{equation}\label{id2}
u=T_bT_a u-T_\rho u\,,
\end{equation}
where $b(x,\xi)\in S_{1/m,\Lambda}(\Omega)$, $\rho(x,\xi)\in S^{-\infty}(\Omega)$ and the operators $T_b$ and $T_\rho$ are properly supported. By Proposition \ref{PROMAXIMAL1}, $T_{a,1}u=T_au$ in the distribution sense. Thus by definition of $T_{a,1}$, $T_au\in L^p_{\rm loc}(\Omega)$. Since $T_b\in \widetilde{\textup{ Op}}S_{1/m,\Lambda}(\Omega)$, it follows from Proposition \ref{SOBOLEVCONT} that $T_bT_au\in H^p_{m,{\rm loc}}(\Omega)$. On the other hand, since $u\in L^p_{\rm loc}(\Omega)$ and $T_\rho$ is a regularizing properly supported operator, $T_\rho u\in H^p_{m,{\rm loc}}(\Omega)$ also follows from Proposition \ref{SOBOLEVCONT}. By \eqref{id2} we then get $u\in H^p_{m,{\rm loc}}(\Omega)$. This gives $D(T_{a,1})\subseteq H^p_{m,{\rm loc}}(\Omega)$ which ends the proof.
\end{proof}

\section{Microlocal properties}\label{MP}
As a direct application of  Proposition \ref{PROPSS2} we have:
\begin{prop}[Regularity of solution to elliptic equations]\label{PROPSS3}
Consider  $p\in]1,\infty[$,  $m(\xi)$, $m'(\xi)$ admissible weights, $a(x,\xi)\in S_{m', \Lambda}(\Omega)$ $m'-$elliptic symbol. Then for every $u\in \mathcal E'(\Omega)$ such that $a(x,D)u\in H^p_{m/m',\textup{loc}} (\Omega)$, we obtain $u\in H^p_{m,\textup{comp}}(\Omega)$. If $a(x, D)$ is properly supported, then $u\in H^p_{m, \textup{loc}}(\Omega)$ for every $u\in\mathcal D'(\Omega)$ such that $a(x,D)u\in H^p_{m/m',\textup{ loc}}(\Omega)$.
\end{prop}
\begin{proof}
We argue directly in the case where $a(x,D)$ is properly supported, the general case being completely analogous. Thanks to Proposition \ref{PROPSS2}, there exists $b(x,D)\in \widetilde{\textup{ Op}}\, S_{1/m', \Lambda}(\Omega)$, such that $b(x,D)a(x,D)={\rm Id}+\rho(x,D)$, with $\rho(x,\xi)\in S^{-\infty}(\Omega)$. Since $\rho(x,D)$ is a regularizing operator and  $a(x,D)u\in H^p_{m/m',\textup{ loc}}(\Omega)$, we can conclude from \eqref{EQSS6} that $u=b(x,D)\left(a(x,D) u\right)-\rho(x,D) u\in H^p_{m,\textup{loc}}(\Omega)$.
\end{proof}
Since the complete lack of any homogeneity property of $\Lambda(\xi)$ and $m(\xi)$ prevents us from defining the characteristic set of $a(x,D)$ in terms of conic neighborhoods in $\mathbb R^n_\xi$,  let us introduce the following alternative tools.\\
The $\Lambda$-neighborhood of a set $X\subset\mathbb R^n$ with length $\varepsilon>0$, is defined as the open set:
\begin{equation}\label{XN}
X_{\varepsilon\Lambda}:=\bigcup_{\xi^0\in X}\left\{\vert\xi_j-\xi^0_j\vert <\varepsilon\lambda_j(\xi^0),  \quad\text{for}\quad j=1,\dots,n\right\}.
\end{equation}
\noindent
Moreover for $x_0\in\Omega$ we set:
\begin{equation}\label{DN}
X(x_0):=\{x_0\}\times X\quad , \quad X_{\varepsilon\Lambda}(x_0):= B_{\varepsilon}(x_0)\times X_{\varepsilon\Lambda},
\end{equation}
where $B_\varepsilon(x_0)$ is the open ball in $\Omega$ centered at $x_0$ with radius $\varepsilon$.
The following properties of  $\Lambda$-neighborhoods can be immediately deduced from \cite[Lemma 1.11]{RO1} (see also \cite{GM2} for an explicit proof). For every $\varepsilon>0$ a suitable $0<\varepsilon^\ast<\varepsilon$, depending only on $\varepsilon$ and $\Lambda$, can be found in such a way that for every $X\subset\mathbb R^n$:
\begin{eqnarray}
&&\left(X_{\varepsilon^\ast\Lambda}\right)_{\varepsilon^\ast\Lambda}\subset X_{\varepsilon\Lambda};\label{eqLAMBDA1}\\
&&\left(\mathbb R^n\setminus X_{\varepsilon\Lambda}\right)_{\varepsilon^\ast\Lambda}\subset\mathbb R^n\setminus X_{\varepsilon^\ast\Lambda}.\label{eqLAMBDA2}
\end{eqnarray}
\begin{defn}\label{MPDEF1}
A symbol $a(x,\xi)\in S_{m,\Lambda}(\Omega)$ is $m-$microlocally elliptic in a set $X\subset\mathbb R^n_\xi$ at the point $x_0\in\Omega$ if there exist positive constants $c_0, R_0$ such that
\begin{equation}\label{ME}
\vert a(x_0, \xi)\vert\geq c_0 m(\xi), \quad \text{when}\quad \xi\in X, \quad \vert\xi\vert>R_0\,.
\end{equation}
We write in this case $a(x,\xi)\in \textup {mce}_{m,\Lambda}(X(x_0))$. By ${\rm Op}\,\textup{mce}_{m,\Lambda}(X(x_0))$ we mean the class of the $m-$pseudodifferential operators with symbol in $\textup {mce}_{m,\Lambda}(X(x_0))$.
\end{defn}
\noindent
According to the notation introduced in Remark \ref{REMWS2}(2), we write respectively ${\rm mce}_{\Lambda}(X(x_0))$ and ${\rm Op}\,\textup{mce}_{\Lambda}(X(x_0))$ for the classes of zero order symbols in $S_{\Lambda}(\Omega)$ that are $m-$microlocally elliptic in $X$ at $x_0$ and the related pseudodifferential operators.
\begin{lemma}\label{MPPROP1}
If a symbol $a(x,\xi)\in \textup{mce}_{m,\Lambda}X(x_0)$ , $x_0\in\Omega$, then there exists a suitable $\varepsilon>0$ such $a(x,\xi)\in\textup{mce}_{m,\Lambda}X_{\varepsilon\Lambda}(x_0)$, that is for suitable constants $C, R>0$
\begin{equation}\label{MPPROP1a}
\vert a(x,\xi)\vert \geq Cm(\xi),\quad \text{for}\, \,(x,\xi)\in X_{\varepsilon\Lambda}(x_0),\quad \vert\xi\vert>R\,.
\end{equation}
\end{lemma}
\begin{proof}
Let the symbol $a(x,\xi)\in S_ {m,\Lambda}(\Omega)$ be $m-$microlocally elliptic in $X\subset\mathbb R^n$ at the point $x_0\in\Omega$ and let $\xi^0\in X$ be arbitrarily fixed. Since $\Omega$ is open, a positive $\varepsilon^\ast$ can be found in such a way that the closed ball $\overline{B}_{\varepsilon^\ast}(x_0)$ is contained in $\Omega$; for $0<\varepsilon<\varepsilon^\ast$ and $(x,\xi)\in X_{\varepsilon\Lambda}(x_0)$, a Taylor expansion of $a(x,\xi)$ centered in $(x_0,\xi^0)$ gives
\begin{equation}\label{taylor1}
\begin{array}{l}
a(x,\xi)-a(x_0,\xi^0)=\sum\limits_{j=1}^n(x^j-x_0^j)\partial_{x^j}a(x_t,\xi^t)dt+(\xi_j-\xi^0_j)\partial_{\xi_j}a(x_t,\xi^t)dt\,,
\end{array}
\end{equation}
where $(x_t,\xi^t):=((1-t)x_0+tx,(1-t)\xi^0+ t\xi)$ for a suitable $0<t<1$. Since $|\xi^t_j-\xi^0_j|=|t||\xi_j-\xi^0_j|<\varepsilon\lambda_j(\xi^0)$ and $|x^{j}_t-x_0^j|=|t||x^j-x^j_0|<\varepsilon$, using \eqref{EQWS1} we can find $C^\ast>0$, depending only on $\varepsilon^\ast$, such that
\begin{equation}\label{taylor2}
\begin{array}{l}
\vert a(x,\xi)-a(x_0,\xi^0)\vert \leq\sum\limits_{j=1}^n \varepsilon C^\ast m(\xi^t)+\varepsilon\lambda_j(\xi^0)C^\ast m(\xi^t)\lambda_j^{-1}(\xi^t)\,.
\end{array}
\end{equation}
In view of \eqref{SW}, \eqref{SW1}, $\varepsilon>0$ can be chosen small enough such that
\begin{equation}\label{taylor3}
\begin{array}{ll}
\frac1{C}\lambda_j(\xi^0)\le\lambda_j(\xi^t)\le C\lambda_j(\xi^0)\,, &1\le j\le n\,,\\
&\\
\frac1{C}m(\xi^0)\le m(\xi^t)\le Cm(\xi^0), &
\end{array}
\end{equation}
for a suitable constant $C>1$ independent of $t$ and $\varepsilon$. Then \eqref{taylor2}, \eqref{taylor3} give
\begin{equation}\label{taylor2.1}
\begin{array}{l}
\vert a(x,\xi)-a(x_0,\xi^0)\vert \leq\hat C\varepsilon m(\xi^0)\,,
\end{array}
\end{equation}
with a suitable constant $\hat C>0$ independent of $\varepsilon$.
\newline
Let the condition \eqref{ME} be satisfied by $a(x,\xi)$ with positive constants $c_0,R_0$. Provided that  $0<\varepsilon<\varepsilon^\ast$ is taken sufficiently small, one can find a positive $R$, depending only on $R_0$, such that $|\xi|>R$ and $|\xi_j-\xi^0_j|<\varepsilon\lambda_j(\xi^0)$ for all $1\le j\le n$ yield $|\xi^0|>R_0$; indeed, from \eqref{PG} we have
\begin{equation*}
|\xi-\xi^0|\le\sum\limits_{j=1}^n|\xi_j-\xi^0_j|<\varepsilon\sum\limits_{j=1}^n\lambda_j(\xi^0)\le nC\varepsilon^\ast(1+|\xi_0|)^C\,,
\end{equation*}
and then
\begin{equation*}
|\xi|\le|\xi^0|+|\xi-\xi^0|\le|\xi^0|+nC\varepsilon^\ast(1+|\xi^0|)^C\,.
\end{equation*}
Hence, it is sufficient to choose $R$ such that $R>R_0+nC\varepsilon^\ast(1+R_0)^C$.
\newline
Since $|\xi^0|>R_0$, the microlocal $m-$ellipticity of $a(x,\xi)$ yields
\begin{equation}\label{ME1}
|a(x_0,\xi^0)|\ge c_0m(\xi^0)\,;
\end{equation}
then \eqref{taylor2.1} and \eqref{ME1} give for $(x,\xi)\in X_{\varepsilon\Lambda}(x_0)$ and $|\xi|>R$
\begin{equation}\label{ME2}
|a(x,\xi)|\ge|a(x_0,\xi^0)|-|a(x,\xi)-a(x_0,\xi^0)|\ge (c_0-\hat C\varepsilon)m(\xi^0)\ge\frac{c_0}{2}m(\xi^0)\,,
\end{equation}
up to a further shrinking of $\varepsilon>0$. From \eqref{ME2}, the condition \eqref{MPPROP1a} follows at once, by using that $m(\xi)\approx m(\xi^0)$.
\end{proof}
\begin{lemma}\label{LEMRO1}
For arbitrary $\varepsilon>0$ and $X\subset\mathbb R^n$ there exists a smooth function  $\sigma(\xi)$ satisfying \eqref{GLOBAL}, such that ${\rm supp}\,\sigma\subset X_{\varepsilon\Lambda}$ and $\sigma(\xi)=1$  if\, $\xi\in X_{\varepsilon^\prime\Lambda}$, for a suitable $\varepsilon^\prime$, satisfying $0<\varepsilon^\prime<\varepsilon$, depending only on $\varepsilon$ and $\Lambda$. Moreover for every $x_0\in\Omega$ there exists a symbol $\tau_0(x,\xi)\in S_{\Lambda}(\Omega)$ such that ${\rm supp}\,\tau_0\subset X_{\varepsilon\Lambda}(x_0)$ and $\tau_0(x,\xi)=1$, for $(x,\xi)\in X_{\varepsilon^\ast\Lambda}(x_0)$, with a suitable $\varepsilon^\ast$ satisfying $0<\varepsilon^\ast<\varepsilon$.
\end{lemma}
For proof see \cite[Lemma 1.10]{RO1}.
\begin{defn}\label{MPDEF2}
We say that a symbol $a(x,\xi)\in S_{m,\Lambda}(\Omega)$ is rapidly decreasing in $\Theta\subset \Omega\times\mathbb R^n$ if there exists $a_0(x,\xi)\in S_{m,\Lambda}(\Omega)$ such that $a(x,\xi)\sim a_0(x,\xi)$ and  $a_0(x,\xi)=0$ in $\Theta$
\end{defn}
\begin{thm}\label{MPTEO1}
For every symbol $a(x,\xi)\in\textup{mce}_{m,\Lambda}X(x_0)$,  $x_0\in\Omega$, there exists a symbol $b(x,\xi)\in S_{1/m,\Lambda}(\Omega)$ such that the associated operator $b(x,D)$ is properly supported and
\begin{equation}\label{MPTEO1a}
b(x,D)a(x,D)={\rm Id}+c(x,D),
\end{equation}
where $c(x,\xi)\in S_{\Lambda}(\Omega)$ is rapidly decreasing in $X_{r\Lambda}(x_0)$ for a suitable $r>0$.
\end{thm}
\begin{proof}
By Proposition \ref{MPPROP1}, there exists $\varepsilon>0$ such that $a(x,\xi)\in\textup{mce}_{m,\Lambda}X_{\varepsilon\Lambda}(x_0)$. Let $\tau_0(x,\xi)$ be a symbol in $S_{\Lambda}(\Omega)$ such that $\tau_0\equiv 1$ on $X_{\varepsilon^\prime\Lambda}(x_0)$, for a suitable $0<\varepsilon^\prime<\varepsilon$, and ${\rm supp}\,\tau_0\subset X_{\varepsilon\Lambda}(x_0)$. We define $b_0(x,\xi)$ by setting
\begin{equation}\label{MPTEO1b}
b_0(x,\xi):=\left\{
\begin{array}{ll}
\displaystyle\frac{\tau_0(x,\xi)}{a(x,\xi)} &\text{for}\quad (x,\xi)\in X_{\varepsilon\Lambda}(x_0),\\
0 &\text{otherwise}\,.
\end{array}
\right.
\end{equation}
Since $a(x,\xi)$ satisfies \eqref{MPPROP1a}, with suitable constants $C, R$, $b_0(x,\xi)$ is a well defined smooth function on the set $\Omega\times\{|\xi|>R\}$. For $k\ge 1$, the functions $b_{-k}(x,\xi)$ are defined recursively on $\Omega\times\{|\xi|>R\}$ by
\begin{equation*}\label{MPTEO1c}
b_{-k}(x,\xi):=\left\{
\begin{array}{ll}
-\sum\limits_{0 <\vert\alpha\vert\leq k}\frac{1}{\alpha !}\partial^\alpha_\xi b_{-k+\vert\alpha\vert}(x,\xi)\dfrac{D^\alpha_x a(x,\xi)}{a(x,\xi)}, &\text{for}\quad (x,\xi)\in X_{\varepsilon\Lambda}(x_0),\\
0 &\text{otherwise}.
\end{array}
\right.
\end{equation*}
The $b_{-k}(x,\xi)$ can be then extended to the whole set $\Omega\times\mathbb R^n$, multiplying them by a smooth cut-off function that vanishes on the set of possible zeroes of the symbol $a(x,\xi)$; for each $k\ge 0$, the extended $b_{-k}$ is a symbol in $S_{\pi^{-k}/m,\Lambda}(\Omega)$. In view of the properties of the symbolic calculus, a symbol $b\in S_{1/m,\Lambda}(\Omega)$ can be chosen in such a way that
\begin{equation*}
b\sim\sum\limits_{k\ge 0}b_{-k}\,,
\end{equation*}
and $b(x,D)$ is a properly supported operator. By construction, the symbol of $b(x,D)a(x,D)$ is equivalent to $\tau_0(x,\xi)$; thus the symbol of $b(x,D)a(x,D)-{\rm Id}$ belongs to $S_{\Lambda}(\Omega)$ and is rapidly decreasing in $X_{r\Lambda}(x_0)$ for $0<r\le\varepsilon^\prime$.
\end{proof}

\begin{defn}\label{DEFLAMBDAF}
A family $\Xi_\Lambda$ of subsets of $\mathbb R^n$ is a $\Lambda$-filter if it is closed with respect to the intersection of any finite number of its elements and moreover:
\begin{eqnarray}
&& X\in \Xi_\Lambda \, \text{and}\,\, X\subset Y,\,\, \text{then}\,\, Y\in \Xi_\Lambda;\label {eqFILTER1}\\
&& \text{for any}\,\, X\in \Xi_\Lambda, \,\, \text{there exists}\,\,\,\, Y\in\Xi_\Lambda\,\,\text{and}\,\, \varepsilon >0\,\, \text{such that}\,\,\,\, Y_{\varepsilon\Lambda}\subset X.\label{eqFILTER2}
\end{eqnarray}
\end{defn}
\begin{defn}\label{MICROREG}
For $X\subset \mathbb R^n$, $x_0\in\Omega$ and $p\in]1,\infty[$ we say that $u\in\mathcal D'(\Omega)$ is microlocally $H^p_m-$regular in $X$ at the point  $x_0\in\Omega$, and write $u\in \textup{mcl}H^{p}_mX(x_0)$, if  there exists a properly supported operator $a(x,D)\in \textup{Op}\,\textup{mce}_{\Lambda}X(x_0)$, such that $a(x,D)u\in H^{p}_{m, {\rm loc}}(\Omega)$.
\end{defn}
\begin{prop}\label{PROPLAMBDAF}
For $u\in\mathcal D'(\Omega)$ and $a(x,\xi)\in S_{m,\Lambda}(\Omega)$, the following families of subsets of $\mathbb R^n$:
\begin{eqnarray}
&&\mathcal W_{m, x_0}^{p}u:=\left\{X\subset\mathbb R^n\, ;\, u\in \textup{mcl} H^{p}_m (\mathbb R^n\setminus X)(x_0)\right\}\,, \quad 1<p<\infty;\label{MPDEF4a}\\
&&\Sigma_{m, x_0}a:=\left\{X\subset\mathbb R^n\, ,\, a(x,\xi)\in \textup{mce}_{m,\Lambda}(\mathbb R^n\setminus X)(x_0)\right\},\label{MPDEF4b}
\end{eqnarray}
are both $\Lambda$-filters.
\end{prop}
We refer to the $\Lambda$-filters in the previous Proposition respectively as:
 \begin{itemize}
\item [-] filter of Sobolev singularities of $u\in\mathcal D'(\Omega)$
\item[-] characteristic filter of $a(x,\xi)\in S_{m, \Lambda}(\Omega)$.
\end{itemize}
\begin{proof}
It is trivial that $\mathcal W^p_{m, x_0}u$ and $\Sigma_{m,x_0}a$ are closed with respect to the intersection of finite number of their elements and satisfy \eqref{eqFILTER1}. Take now $X\in \Sigma_{m,x_0}a$. Using Lemma \ref{MPPROP1} we have, for some $\epsilon>0$, $a(x,\xi)$ $m-$microlocally elliptic in $(\mathbb R^n\setminus X)_{\epsilon\Lambda}(x_0)$. Set $Y=\mathbb R^n\setminus(\mathbb R^n\setminus X)_{\epsilon\Lambda}$. By means of \eqref{eqLAMBDA2} we have for some $0<\varepsilon<\epsilon$:
\begin{equation}\label{eqPROPLAMDAF1}
Y_{\varepsilon\Lambda}=(\mathbb R^n\setminus(\mathbb R^n\setminus X)_{\epsilon\Lambda})_{\varepsilon\Lambda}\subset \mathbb R^n\setminus(\mathbb R^n\setminus X)_{\varepsilon\Lambda}\subset X.
\end{equation}
The same can be easily verified for  $\mathcal W^p_{m, x_0}u$ and the proof is then concluded.
\end{proof}
Concerning the filter of Sobolev singularities of a distribution, the following result can be proven by arguing similarly to \cite[Proposition 2.8]{RO1}.
\begin{prop}\label{PROPLAMBDAF1}
The following conditions are equivalent:
\begin{itemize}
\item[{\it a}.] $\emptyset\in\mathcal W^p_{m, x_0}u$;
\item[{\it b}.] there exists $\phi\in C^\infty_0(\Omega)$, with $\phi(x_0)\neq 0$, such that $\phi u\in H^p_m$;
\item[{\it c}.] there exist $X_1, \dots, X_H\subset\mathbb R^n$, with $\bigcup\limits_{h=1}^H X_h=\mathbb R^n$, such that\\
$u\in\textup{mcl} H^p_mX_h(x_0)$ for $h=1,\dots, H$.
\end{itemize}
\end{prop}
\begin{proof}
Assume that $\emptyset\in\mathcal W^p_{m, x_0}u$. Then there exists  $a(x,D)\in\widetilde{\textup{Op}}S_{\Lambda}(\Omega)$, with microlocally elliptic symbol in the set $\{x_0\}\times\mathbb R^n$, such that $a(x,D)u\in H^p_{m, {\rm loc}}(\Omega)$. Thanks to Theorem \ref{MPTEO1}, there also exists $b(x,D)\in\widetilde{\textup{Op}}S_{\Lambda}(\Omega)$ such that
\begin{equation}\label{eq:1}
u=b(x,D)a(x,D)u-c(x,D)u\,,
\end{equation}
where $c(x,D)\in\widetilde{\textup{Op}}S_{\Lambda}(\Omega)$ has rapidly decreasing symbol in a set $B_r(x_0)\times\mathbb R^n$ with a sufficiently small $r>0$. Take a function $\phi\in C^\infty_0(\Omega)$, with ${\rm supp}\,\phi\subset B_r(x_0)$ and $\phi\equiv 1$ on $B_{r/2}(x_0)$, such that
\begin{equation}\label{eq:2}
\phi(x)c(x,\xi)\in S^{-\infty}(\Omega)\,,
\end{equation}
and let $\tilde\phi\in C^\infty_0(\Omega)$ be chosen such that
\begin{equation}\label{eq:3}
\phi c(x,D)u=\phi c(x,D)(\tilde\phi u)\,.
\end{equation}
In view of Proposition \ref{SOBOLEVCONT} (and using in particular that $\phi c(x,D)(\tilde\phi\cdot)$ is a regularizing operator), from the equations \eqref{eq:1}-\eqref{eq:3} we get
\begin{equation}
\phi u=\phi b(x,D)a(x,D)u-\phi c(x,D)(\tilde\phi u)\in H^p_m\,,
\end{equation}
hence condition $b$ is satisfied.
\newline
Conversely, assume that  $\phi u\in H^p_m$, for a suitable function $\phi\in C^\infty_0(\Omega)$, $\phi(x_0)\neq 0$; then we have that $(a.)$ holds true by considering the function $\phi(x)$ itself as a symbol in $S_\Lambda(\Omega)$ microlocally elliptic in $\{x_0\}\times\mathbb R^n$.
\newline
The equivalence between conditions $a$ and $c$ follows at once from the general properties of the filters.
\end{proof}
We can then state the main result of this section:
\begin{thm}\label{MAINTHM}
For $m$, $m'$ arbitrary admissible weights, associated to the same  weight vector $\Lambda$, consider $a(x,D)\in\widetilde{\textup{Op}}\, S_{m,\Lambda}(\Omega)$, $x_0\in\Omega$, $p\in ]1,\infty[$. Then for any $u\in\mathcal D'(\Omega)$ we have:
\begin{equation}\label{eqMAINTHM1}
\mathcal W^p_{m'/m, x_0}a(x,D)u\cap \Sigma_{x_0}a\subset \mathcal W^p_{m', x_0} u\subset \mathcal W^p_{m'/m, x_0}a(x,D)u.
\end{equation}
\end{thm}
The proof directly follows from two next propositions.
\begin{prop}\label{cont1}
Let $x_0\in\Omega$, $X\subset\mathbb{R}^n$, $a(x,D)\in\widetilde{\textup{Op}}S_{m,\Lambda}(\Omega)$ be given. Then for $p\in]1,\infty[$ and $u\in \textup{mcl}H^p_{m^\prime}X(x_0)$ one has $a(x,D)u\in \textup{mcl}H^{p}_{m^\prime/m}X(x_0)$.
\end{prop}
\begin{proof}
From Definition \ref{MICROREG}, there exists $b(x,D)\in\widetilde{\textup{Op}}S_{\Lambda}(\Omega)$, with $m-$microlocally elliptic symbol, such that $b(x,D)u\in H^p_{m^\prime, {\rm loc}}(\Omega)$. From Theorem \ref{MPTEO1} there also exists an operator $c(x,D)\in \widetilde{{\textup{Op}}}S_{\Lambda}(\Omega)$ such that
\begin{equation}\label{conteqn1}
c(x,D)b(x,D)={\rm Id}+\rho(x,D)\,,
\end{equation}
where $\rho(x,\xi)\in S_{\Lambda}(\Omega)$ is rapidly decreasing in $X_{r\Lambda}(x_0)$ for some $0<r<1$. Let $r^\ast>0$ be such that
\begin{equation}\label{intorni}
(\mathbb R^n\setminus X_{r\Lambda})_{r^\ast\Lambda}\subset\mathbb R^n\setminus X_{r^\ast\Lambda}\,,\qquad 0<r^\ast<r\,,
\end{equation}
and take a symbol $\tau_0(x,\xi)\in S_{\Lambda}(\Omega)$ satisfying
\begin{equation*}
{\rm supp}\,\tau_0\subset X_{r^\ast\Lambda}(x_0)\,,\qquad\tau_0\equiv 1\,\,{\rm on}\,\,X_{r^\prime\Lambda}(x_0)\,,
\end{equation*}
with a suitable $0<r^\prime<r^\ast$. Finally, let $\tau(x,\xi)$ be a symbol such that $\theta_0(x,\xi):=\tau(x,\xi)-\tau_0(x,\xi)\in S^{-\infty}(\Omega)$ and $\tau(x,D)\in\widetilde{\textup{Op}}S_{\Lambda}(\Omega)$. One can check (see \cite{GM2} for details) that $\tau(x,\xi)$ is $m-$microlocally elliptic in $X$ at $x_0$; in particular, $\tau(x,\xi)=\theta_0(x,\xi)\in S^{-\infty}(\Omega)$ for $(x,\xi)\notin X_{r^\ast\Lambda}(x_0)$.
\newline
Arguing as in the proof of \cite[Theorem 2]{GM2}, from \eqref{conteqn1} we write
\begin{equation*}\label{conteqn2}
\tau(x,D)a(x,D)u=\tau(x,D)a(x,D)c(x,D)(b(x,D)u)-\tau(x,D)a(x,D)\rho(x,D)u\,.
\end{equation*}
Since $\tau(x,D)a(x,D)c(x,D)\in\widetilde{\textup{Op}}S_{m,\Lambda}(\Omega)$ and $b(x,D)u\in H^p_{m^\prime, {\rm loc}}(\Omega)$ then we get  $\tau(x,D)a(x,D)c(x,D)(b(x,D)u)\in H^p_{m^\prime/m, {\rm loc}}(\Omega)$. Moreover, it can be shown that in view of \eqref{intorni}
\begin{equation*}
\varphi(x)\tau(x,D)a(x,D)\rho(x,D)u\in C^\infty_0(\Omega)\subset H^p_{m^\prime/m}\,,
\end{equation*}
for every $\varphi\in C^\infty_0(\Omega)$, so that $\tau(x,D)a(x,D)\rho(x,D)u\in H^p_{m^\prime/m, {\rm loc}}(\Omega)$ (for the details see \cite[Theorem 2]{GM2}).
\newline
This proves that $\tau(x,D)a(x,D)u\in H^p_{m^\prime/m, {\rm loc}}(\Omega)$ and ends the proof.
\end{proof}
\begin{prop}\label{cont2}
For $x_0\in\Omega$, $X\subset\mathbb{R}^n$, let the symbol of $a(x,D)\in\widetilde{\textup{Op}}S_{m,\Lambda}(\Omega)$ be $m-$microlocally elliptic in $X$ at the point $x_0$. Then for every $p\in]1,\infty[$ and $u\in\mathcal{D}'(\Omega)$ such that $a(x,D)u\in\textup{mcl}H^{p}_{m^\prime/m}X(x_0)$ one has $u\in \textup{mcl}H^{p}_{m^\prime}X(x_0)$.
\end{prop}
\begin{proof}
From Definition \ref{MICROREG}, there exists $b(x,D)\in\widetilde{\textup{Op}}S_{\Lambda}(\Omega)$ $m-$microlocally elliptic in $X$ at $x_0$ such that
\begin{equation}\label{cont2eqn1}
b(x,D)a(x,D)u\in H^p_{m^\prime/m, {\rm loc}}(\Omega)\,.
\end{equation}
From Theorem \ref{MPTEO1} there exist $c(x,D)\in\widetilde{\textup{Op}}S_{\Lambda}(\Omega)$ and $q(x,D)\in\widetilde{\textup{Op}}S_{1/m,\Lambda}(\Omega)$  such that
\begin{equation}\label{cont2eqn2}
c(x,D)b(x,D)=\text{Id}+\rho(x,D)\,,\qquad q(x,D)a(x,D)=\text{Id}+\sigma(x,D)\,,
\end{equation}
with $\rho(x,\xi), \sigma(x,\xi)\in S^{-\infty}(\Omega)$ rapidly decreasing in $X_{r\Lambda}(x_0)$ for a suitable $0<r<1$.
\newline
Let the symbols $\tau_0(x,\xi), \tau(x,\xi)\in S_{\Lambda}(\Omega)$ be constructed as in the proof of Proposition \ref{cont1}. It can be proved that $\tau(x,D)u\in H^p_{m, {\rm loc}}(\Omega)$, by writing
\begin{equation}\label{cont2eqn3}
\begin{array}{ll}
\tau(x,D)u=\tau(x,D)q(x,D)c(x,D)\left(b(x,D)a(x,D)u\right)\\
\\
\quad -\tau(x,D)q(x,D)\rho(x,D)a(x,D)u-\tau(x,D)\sigma(x,D)u\,,
\end{array}
\end{equation}
where the identities \eqref{cont2eqn2} have been used, and applying similar arguments as in the proof of Proposition \ref{cont1} (see also \cite[Theorem 3]{GM2}).
\end{proof}
\section{Some examples}\label{EX}
For  $M=(1,2)$ we can  define in $\mathbb R^2$  the \textit{quasi-homogeneous} weight function $\langle \xi\rangle_{M}=\left(1+\xi_1^2+\xi^4_2 \right)^{1/2}$ and the weight vector $\Lambda_{M}(\xi)=\left(\langle \xi\rangle_{M}, \langle\xi\rangle_{M}^{1/2}\right)$. For simplicity of notation we set $H^p_{\langle\cdot\rangle_M^s}=H^{s, p}_M$, $s\in \mathbb R$.
\begin{ex}\label{EXEX1}
Let us introduce the following  operators
\begin{eqnarray}
&P(x,\partial)=x_1\partial_{x_1}-\partial_{x^2_2} &\text{with symbol}\quad  p(x,\xi)=ix_1\xi_1+\xi_2^2; \label{eqEX1}\\
&Q(x,\partial)= x_1\partial_{x_1}+ i\partial_{x_1}-\partial_{x_2^2}  &\text{with symbol}\quad  q(x,\xi)=ix_1\xi_1-\xi_1+\xi_2^2.\label{eqEX2}
\end{eqnarray}
Let us set $\Omega:=\mathbb R^2\setminus\{x_1=0\}$. As operators in  $\textup{Op}\,S_{\langle \cdot\rangle_M, \Lambda_M}(\Omega)$, $P(x,\partial)$ and $Q(x,\partial)$ are $m$-elliptic with respect to the weight function $m(\xi)=\langle \xi\rangle_M$; we say in this case that they are  $M$- elliptic. Thus for both the operators the  minimal and maximal extensions  in $L^p_\textup{loc}(\Omega)$ coincide and their domain is the local Sobolev space $H^{1,p}_{M,\textup{loc}}(\Omega)$.\\
As operator in $S_{\langle \cdot\rangle_M, \Lambda_M}(\mathbb R^2)$, $Q(x,\partial)$ is $M-$microlocally elliptic at every point $x^0=(0, x^0_2)$, with arbitrary $x^0_2\in\mathbb R$, in any set of the family
\begin{equation}\label{eqEX2.1}
X_k=\left\{(\xi_1,\xi_2)\in \mathbb R^2\,\,;\,\, \xi_1<(1-k)\xi_2^2\,\,\textup{or}\,\, \xi_1>\frac{1}{1-k}\xi_2^2 \right\},\quad 0<k<1,
\end{equation}
and we write $Q(x,\partial)\in \textup{mqe}\, X_k(x^0)$.
Thus the characteristic filter $\Sigma_{x^0}q$ admits as base of filter the family of sets
\begin{equation}\label{eqEX3}
(\mathbb R^2\setminus X_k)=\left\{ (1-k)\xi_2^2\leq\xi_1\leq\frac1{1-k}\xi_2^2\right\}, \quad 0<k<1.
\end{equation}
Applying  Theorem \ref{MAINTHM} we obtain that any solution to the equation  $Q(x,\partial)u=f$, with $f\in H^{s,p}_{M, \textup{loc}}(\mathbb R^2)$ is microlocally $H^{s+1,p}_M$ regular in any set $\{x^0\}\times X_k$, with $x^0=(0,x^0_2)$ and $0<k<1$.

Notice now that $p(x,\xi)$, $q(x,\xi)$ may be considered  as \textit{principal symbols} of $P(x,\partial)$, $Q(x,\partial)$ with respect to the quasi-homogeneous weight $\langle\cdot\rangle_M$. Introducing now the $M$-characteristic set of $Q(x,\partial)$:
\begin{equation}\label{eqEX4}
\textup{Char}_M Q=\left\{ (x,\xi)\in\mathbb R^2_x\times\mathbb R^2_\xi\setminus\{0\}\, ,\, q(x,\xi)=0\right\},
\end{equation}
we have 
\begin{equation*}\textup{Char}_M Q=\left\{ (0,x_2,\xi_1,\xi_2);\,\,\, x_2\in\mathbb R\,,\,\,\xi_1=\xi_2^2\,,\,\,\xi_2\neq 0 \right\}=\{0\}\times\mathbb R\times\bigcap\limits_{0<k<1}(\mathbb R^2\setminus X_k).
\end{equation*}

Let us observe at the end that any set $X_k$ is $M$-conic, that is for any $\xi\in X_k$ and $t>0$ we have $t^{1/M}\xi:=(t\xi_1, t^{1/2}\xi_2)\in X_k$. The same clearly holds for $\mathbb R^2\setminus X_k$ and Char$_{M} Q$, agreeing with \cite [Definition 4.3]{GM3}.

About the operator $P(x,\partial)$ we can notice that for $x^0=(0,x^0_2)$, with an arbitrary $x^0_2\in\mathbb R$, $p(x^0,\xi)=0$ if and only if $\xi_2=0$; thus the $M$-characteristic set Char$_M P=\{(0, x_2, \xi_1, 0);\,\,\,\xi_1\neq 0\}$ coincides with the classical (conic) characteristic set Char $P$.
\end{ex}

Notice moreover that the heat operator $\partial_{x_1}-\partial_{x_2^2}$ is clearly $M$-elliptic, as operator in $S_{\langle\cdot \rangle_M,\Lambda_M}(\mathbb R^2)$, while the Schroedinger operator $i\partial _{x_1}-\partial_{x_2^2}$ is $M-$microlocally elliptic at any point $x_0\in \mathbb R^2$ in the family of set expressed in \eqref{eqEX2.1}.
\smallskip

Let us define now the positive function in $\mathbb R^2$
\begin{equation}\label{eqEX5}
\lambda(\xi):=(1+\xi_1^6+\xi_1^4\xi_2^4+\xi_2^6)^{1/2},
\end{equation}
which may be considered  as \textit{multi-quasi-homogeneous} weight in Example \ref{EXWS1} (2). Precisely $\lambda(\xi)=\left(\sum_{\alpha\in V(\mathcal P)}\xi^{2\alpha}\right)^{1/2}$, where $V(\mathcal P)=\left\{(0,0), (3,0), (2,2), (0,3)   \right\}$ are the vertices of a complete Newton polyhedron. Roughly speaking the normal vector to any face of $\mathcal P$ not laying on the coordinate axes has not zero components, see \cite[\S 1.1]{BBR}, \cite{GM1}. Moreover again from \cite[\S 1.1]{BBR}, we have that the formal order of $\lambda(\xi)$ is $\mu=6$. Then $\Lambda(\xi)=\left(\lambda(\xi)^{1/6}+\vert\xi_1\vert, \lambda(\xi)^{1/6}+\vert\xi_2\vert\right)$ is a weight vector and, for any $r\in\mathbb R$, $\lambda(\xi)^r$ is an admissible weight, see example \ref{EXWS1} (2).
\begin{ex}\label{EXEX2}
Consider the linear partial differential operators in $\textup{Op}\, S_{\lambda, \Lambda}(\mathbb R^2)$:
\begin{eqnarray}
&&A(x, \partial)=(x_1\partial_{x_1}-\partial_{x_2^2})(x_2\partial_{x_2}-\partial_{x_1^2})\label{eqEX6},\\
&& B(x,\partial)= (x_1\partial _{x_1}+i\partial_{x_1}-\partial_{x_2^2})(x_2\partial_{x_2}+i\partial_{x_2}-\partial_{x_1^2}).\label{eqEX8}
\end{eqnarray}
According to the rules of the symbolic calculus (see Proposition \ref{PROMM1}), their symbols $a(x,\xi)$ and $b(x,\xi)$ can be written in the form
\begin{eqnarray}
&&a(x,\xi)=(ix_1\xi_{1}+\xi_2^2)(ix_2\xi_{2}+\xi_{1}^2)+2\xi_2^2\label{eqEX7},\\
&& b(x,\xi)=(ix_1\xi_{1}-\xi_{1}+\xi_{2}^2)(ix_2\xi_{2}-\xi_{2}+\xi_{1}^2)+2\xi_2^2.\label{eqEX9}
\end{eqnarray}
The term $2\xi_2^2$ that appears in the right-hand side of both formulas \eqref{eqEX7}, \eqref{eqEX9} behaves as a lower order symbol with respect to the weight function \eqref{eqEX5}. Hence from the ellipticity properties of $P(x,\partial)$ and $Q(x,\partial)$ collected in the example \ref{EXEX1}, we find that $A(x,\partial)$ and $B(x, \partial)$ are $\lambda$-elliptic as operators in Op$S_{\lambda, \Lambda}(\Omega)$, where we have set $\Omega:=\mathbb R^2\setminus\bigcup\limits_{j=1,2}\{x_j=0\}$. Applying Theorem \ref{ADN} we obtain in both cases that the minimal and maximal extensions of those operators in $L^p_{\textup{loc}}(\Omega)$ coincide, with domain given by $H^p_{\lambda, \textup{loc}}(\Omega)$.

As operators in Op$S_{\lambda,\Lambda}(\mathbb R^2)$, $A(x,\partial)$ and $B(x,\partial)$ fail to be $\lambda-$elliptic at the points of the coordinate axes. Using again the results about the operator $Q(x,\partial)$ given in the example \ref{EXEX1}, the behavior of the operator $B(x,\partial)$ along the coordinate axes can be summarized as follows.
\begin{itemize}
\item[i.] At any point $x^0=(0,x^0_2)$ (with an arbitrary $x^0_2\neq 0$), the symbol of $B(x,\partial)$ reduces to $b(x^0,\xi)=(-\xi_{1}+\xi_{2}^2)(ix^0_2\xi_{2}-\xi_{2}+\xi_{1}^2)+2\xi_2^2$, hence it is $\lambda-$microlocally elliptic in all sets $X_k$ of the type \eqref{eqEX2.1} with arbitrary $0<k<1$. This also means that the characteristic filter $\Sigma_{x^0}b$ admits as base the family of sets:
\begin{equation*}
\left\{ \xi\in\mathbb R^2\quad ;\quad (1-k)\xi_2^2\leq\xi_1\leq \frac{1}{1-k}\xi_2^2\right\}_{0<k<1}\,.
\end{equation*}
\end{itemize}
Arguing similarly on the other points of the coordinate axes, we obtain that:
\begin{itemize}
\item[ii.] at any point $y^0=(y^0_1,0)$ (with an arbitrary $y^0_1\neq 0$), the characteristic filter $\Sigma_{y^0}b$ admits as base the family of sets:
\begin{equation*}
\left\{ \xi\in\mathbb R^2\quad ;\quad (1-k)\xi_1^2\leq\xi_2\leq \frac{1}{1-k}\xi_1^2\right\}_{0<k<1}\,;
\end{equation*}
\item[iii.] at the origin ${\bf 0}=(0,0)$, the characteristic filter $\Sigma_{\bf 0}b$ admits as base the family of sets:
\begin{equation*}
\left\{ \xi\in\mathbb R^2\quad ;\quad
\begin{array}{l}
 (1-k)\xi_2^2\leq\xi_1\leq \frac{1}{1-k}\xi_2^2\\
\qquad\qquad\,\,\,\mbox{or}\\
 (1-k')\xi_1^2\leq\xi_2\leq \frac{1}{1-k'}\xi_1^2
\end{array}
\right\}_{0<k, k'<1} .
\end{equation*}
\end{itemize}
Applying  Theorem \ref{MAINTHM} we obtain that for any solution $u$ of the equation $B(x,\partial)u=f\in H^p_{\lambda^s,\textup{loc}}(\mathbb R^2)$, the filter of Sobolev singularities $\mathcal W^p_{\lambda^{s+1}, x^0} u$ at every point $x^0=(x^0_1,x^0_2)$ belonging to the coordinate axes (that is such that $x^0_1x^0_2=0$) contains the characteristic filter $\Sigma_{x^0} b$.
\end{ex}

\end{document}